\documentclass[11pt,a4paper]{article}
\usepackage{epsfig,amsfonts,amsgen,amsmath,amstext,amsbsy,amsopn,amsthm
}

\usepackage{color}
\newtheorem{theorem}{Theorem}

\newtheorem{prop}{Proposition}
\newtheorem{lemma}{Lemma}
\newtheorem{false statement}{False statement}

\theoremstyle{definition}

\newtheorem{case}{Case}

\newcounter{mathitem}
  {\begin{list}{{$(\roman{mathitem})$}}{
   \setcounter{mathitem}{0}
   \usecounter{mathitem}
   \setlength{\topsep}{0pt plus 2pt minus 0pt}
   \setlength{\parskip}{0pt plus 2pt minus 0pt}
   \setlength{\partopsep}{0pt plus 2pt minus 0pt}
   \setlength{\parsep}{0pt plus 2pt minus 0pt}
   \setlength{\leftmargin}{35pt}
   \setlength{\itemsep}{0pt plus 2pt minus 0pt}}}
  {\end{list}}

\begin{document}

\title{\bf\Large Long paths and cycles passing through specified vertices {under} the average degree condition}

%\date{}

\author{Binlong Li, Bo Ning and Shenggui Zhang\thanks{Corresponding author.
E-mail address: sgzhang@nwpu.edu.cn (S. Zhang).}\\[2mm]
\small Department of Applied Mathematics,
\small Northwestern Polytechnical University,\\
\small Xi'an, Shaanxi 710072, P.R.~China} \maketitle

\begin{abstract}
{Let $G$ be a $k$-connected graph with $k\geq 2$. In this paper we first prove that: For two distinct vertices
$x$ and $z$  in $G$, it contains a path
passing through its any $k-2$ {specified} vertices with length at least the
average degree of the vertices other than $x$ and $z$. Further, with this result, we prove that: If $G$ has $n$ vertices and $m$ edges, then it contains a cycle of length at least $2m/(n-1)$ passing through its any $k-1$ specified vertices. Our results generalize a theorem of Fan on the existence of long paths  and a classical theorem of Erd\"os and Gallai on the existence of long cycles under the average degree condition.}

\smallskip
\end{abstract}

\section{Introduction}

We use Bondy and Murty \cite{Bondy_Murty} for terminology and
{notations} not defined here and consider finite simple graphs only.

Let $G$ be a graph {and $H$ a subgraph of $G$. We} use $V(H)$ and $E(H)$ to
denote the {set} of vertices and edges of $H$, respectively, and use
$e(H)$  { for the number of the edges of $H$}. For a vertex $v\in V(G)$,
$N_H(v)$ {denotes} the set, and $d_H(v)$ the number, of neighbors of $v$ in
$H$. We call $d_H(v)$ the \emph{degree} of $v$ in $H$. Let $x$ and $z$ be two distinct vertices of $G$. A path connecting
{$x$ and $z$} {is} called an $(x,z)$-\emph{path}. For a
subset $Y$ of $V(G)$, an $(x,z)$-path passing through all the
vertices in $Y$ is {called} an $(x,Y,z)$-\emph{path}, and a cycle passing
through all the vertices in $Y$ is {called} a $Y$-\emph{cycle}. If $Y$
contains only one vertex $y$, an $(x,\{y\},z)$-path and a
$\{y\}$-cycle are simply denoted by an $(x,y,z)$-path and a
$y$-cycle, respectively. The \emph{distance} between $x$ and $z$ in $H$, denoted by
$d_H(x,z)$, is {the length of a shortest} $(x,z)$-path with all its
internal vertices in $H$. If no such a path exists, we define
$d_H(x,z)=\infty$. The {\emph{codistance}} between $x$ and $z$ in $H$,
denoted by $d^*_H(x,z)$, is {the length of a longest} $(x,z)$-path
with all its internal vertices in $H$. If no such a path exists, we
define $d^*_H(x,z)=0$. When no confusion occurs, we use $N(v)$, $d(v)$, $d(x,z)$ and
$d^*(x,z)$ instead of $N_G(v)$, $d_G(v)$, $d_G(x,z)$ and $d^*_G(x,z)$, respectively.

{Long path and cycle problems are interesting and important in graph theory
and have been deeply studied}, see {\cite{Bondy,Gould}.} The following Theorem by Erd\"os and Gallai opened the
study on long paths with specified end vertices.

\begin{theorem}[Erd\"os and Gallai \cite{Erdos_Gallai}]
Let $G$ be a 2-connected graph and $x$ and $z$ be two distinct
vertices of $G$. If $d(v)\geq d$ for every {vertex} $v\in
V(G)\backslash\{x,z\}$, then $G$ contains an $(x,z)$-path of
length at least $d$.
\end{theorem}

In fact, Theorem 1 has a stronger {extension due to} Enotomo.

\begin{theorem}[Enotomo \cite{Enomoto}]
Let $G$ be a 2-connected graph and $x$ and $z$ be two distinct
vertices of $G$. If $d(v)\geq d$ for every {vertex} in
$V(G)\backslash\{x,z\}$, then for every given vertex $y\in
V(G)\backslash\{x,z\}$, $G$ contains an $(x,y,z)$-path of length at
least $d$.
\end{theorem}

Another direction of extending Theorem 1 is to weaken the minimum
degree condition to an average degree condition. Fan finished this
work as follows.

\begin{theorem}[Fan \cite{Fan}]
Let $G$ be a 2-connected graph and $x$ and $z$ be two distinct
vertices of $G$. If the average degree of the vertices other than
$x$ and $z$ is at least $r$, then $G$ contains an $(x,z)$-path {of}
length at least $r$.
\end{theorem}

The following graph {shows} that {one} cannot {replace} the
minimum degree condition in Theorem 2 by {the} average degree condition.
Let $H$ be a complete graph on $n-1$ vertices and $x,z\in V(H)$. Let
$G$ be a graph obtained from $H$ by adding a new vertex $y$ and two
edges $xy,yz$. Then the {length of the longest $(x,y,z)$-path in $G$ is
2}, less than the average degree of the vertices other than $x$ and
$z$ when $n\geq 5$.

In this paper, we first generalize Theorem 3 to $k$-connected graphs
and get the following result.

\begin{theorem}
Let $G$ be a $k$-connected graph with $k\geq 2$, and $x$ and $z$
be two distinct vertices of $G$. If the average degree of the
vertices other than $x$ and $z$ is at least $r$, then for any subset $Y$ of
$V(G)$ with $|Y|= k-2$, $G$ contains an $(x,Y,z)$-path of
length at least $r$.
\end{theorem}

We {postpone} the proof of Theorem 4 to Section 3.

Now we consider {long cycles} passing through specified vertices
in graphs. Theorem 5 shows the existence of long cycles {in}
2-connected graph {under the} minimum degree condition, and Theorem 6 extends
Theorem 5 to $k$-connected graphs.

\begin{theorem}[{Locke \cite{Locke}}]
Let $G$ be a 2-connected graph. If the minimum degree of $G$ is at
least $d$, then for any two vertices $y_1$ and $y_2$ of $G$, $G$
contains either a $\{y_1,y_2\}$-cycle of length at least $2d$ or a
Hamilton cycle.
\end{theorem}

\begin{theorem}[Egava, Glas and Locke \cite{Egava_Glas_Locke}]
Let $G$ be a $k$-connected graph with $k\geq 2$. If the minimum
degree of $G$ is at least $d$, then for any subset $Y$ of $V(G)$
with $|Y|=k$, $G$ contains either a $Y$-cycle of length at least
$2d$ or a Hamilton cycle.
\end{theorem}

On the existence of long cycles in graphs with a given {number of edges},
Erd\"os and Gallai gave the following {result}.

\begin{theorem}[Erd\"os and Gallai \cite{Erdos_Gallai}]
Let $G$ be a 2-edge-connected graph on $n$ vertices. Then $G$
contains a cycle of length at least $\frac{2e(G)}{n-1}$.
\end{theorem}

In this paper, as an application of {{Theorems 4}}, we give {the
following theorem} on long cycles passing through specified
vertices of graphs with a given {number of edges}.

\begin{theorem}
Let $G$ be a $k$-connected graph on $n$ vertices with $k\geq 2$.
Then for any subset $Y$ of $V(G)$ with $|Y|=k-1$, $G$ contains a
$Y$-cycle of length at least $\frac{2e(G)}{n-1}$.
\end{theorem}

In Theorem 8, one cannot expect a cycle passing through $k$ specified vertices of length at least $2e(G)/(n-1)$. Let $H$ be a complete graph on $n-k$ vertices with $n>3k$ and $u_1,u_2,\ldots,u_k$ be $k$ vertices of $H$. Let $Y=\{v_1,v_2,\ldots,v_k\}$ be a set of vertices not in $V(H)$. We construct a graph $G$ with $V(G)=V(H)\cup Y$ and $E(G)=E(H)\cup\{u_iv_j:1\leq i,j\leq k\}$. Then $G$ is a $k$-connected graph and the longest $Y$-cycle has length $2k$, which is less than
$$
\frac{2e(G)}{n-1}=\frac{(n-k)(n-k-1)+2k^2}{n-1}.
$$

We {postpone} the
proof of Theorem 8 in Section 4.

\section{Preliminaries}

{Let $G$ be a graph and {$P$, $H$} two disjoint subgraphs of $G$.
We use $E(P,H)$ to denote the set, and $e(P,H)$ the number, of edges
{with one vertex in $P$ and the other in $H$}. If
{$E(P,H)\neq\emptyset$}, then we call $P$ and $H$ are \emph{joined}.
We use $N_P(H)$ to denote the set of vertices in $P$ which are joined
to $H$. If {$x$ is a vertex} in $G-P$, we say that
$x$ is \emph{locally $k$-connected} to $P$ (in $G$) if there are $k$
paths connecting $x$ to vertices in $P$ such that any two of them
have only the vertex $x$ in common. We say that $H$ is \emph{locally $k$-connected} to $P$ (in
$G$) if for every vertex $x\in V(H)$, $x$ is locally $k$-connected
to $P$. Note that if $H$ is locally $k$-connected to $P$, then $H$
is locally $l$-connected to $P$ for all $l$, $0\leq l\leq k$; and,
if $G$ is $k$-connected and $|V(P)|\geq k$, then $H$ is locally $k$-connected to $P$ in
$G$.}

The following propositions on {local} $k$-connectedness are proved
in \cite{Fan}.

\begin{prop}[Fan \cite{Fan}]
Let $H$ and $P$ be two disjoint subgraphs of a graph $G$. If $H$ is
locally $k$-connected to $P$ in the subgraph induced by $V(H)\cup
V(P)$, then $E(P,H)$ contains an independent set of $t$ edges, where
$t\geq\min\{k,|V(H)|\}$.
\end{prop}

\begin{prop}[Fan \cite{Fan}]
Let $H$ and $P$ be two disjoint subgraphs of a graph G. Let $u\in
N_P(H)$ and $G'$ be the graph obtained from $G$ by deleting all
edges from $u$ to $H$. If $H$ is locally $k$-connected to $P$ in
$G$, then $H$ is locally $(k-1)$-connected to $P$ in $G'$.
\end{prop}

\begin{prop}[Fan \cite{Fan}]
Let $H$ and $P$ be two disjoint subgraphs of a graph $G$, and $B$ a
block of $H$. Let $H'$ be the subgraph obtained from $H$ by
contracting $B$. If $H$ is locally $k$-connected to $P$ in $G$, then
$H'$ is also locally $k$-connected to $P$ in the resulting graph.
\end{prop}

Next we introduce the concept of local maximality for paths.

Let $P$ be a path of a graph $G$, and $u,v\in V(P)$. We use $P[u,v]$
to denote the segment of $P$ from $u$ to $v$, and $P(u,v)$ the
segment obtained from $P[u,v]$ by deleting the two end vertices $u$
and $v$. Let $H$ be a component of $G-P$. We say that $P$ is a
\emph{locally longest path with respect to} $H$ if we cannot obtain
a longer path than $P$ by replacing {the} segment $P[u,v]$ by a
$(u,v)$-path with all its internal vertices in $H$. In other words,
$P$ is locally longest with respect to $H$ if, for any $u,v\in
V(P)$,
$$
e(P[u,v])\geq d_H^*(u,v).
$$
{If $P$ is an $(x,Y,z)$-path of $G$, where $x,z\in V(G)$ and
$Y\subset V(G)$, then we} say that $P$ is a \emph{locally longest
$(x,Y,z)$-path with respect to} $H$ if we cannot obtain a longer
$(x,Y,z)$-path than $P$ by replacing {the} segment $P[u,v]$ with $Y\cap
V(P(u,v))=\emptyset$ by a $(u,v)$-path with all its internal
vertices in $H$. Note that if $P$ is a longest path (longest $(x,Y,z)$-path) in a
graph $G$, then, of course, $P$ is a locally longest path (locally
longest $(x,Y,z)$-path) with respect to any component of $G-P$. If two
vertices $u$ and $u'$ in $V(P)$ are joined to $H$ by two independent
edges, then we call $\{u,u'\}$ a \emph{strong attached pair} of $H$
to $P$. A \emph{strong attachment} of $H$ to $P$ (in $G$) is a
subset $T=\{u_1,u_2,\ldots,u_t\}\subset N_P(H)$, where $u_i$, $1\leq
i\leq t$, are in order {along} $P$, such that each {ordered} pair
$\{u_i,u_{i+1}\}$, $1\leq i\leq t-1$, is a strong attached pair of
$H$ to $P$. A strong attachment $T$ of $H$ to $P$ is \emph{maximum}
if it has maximum cardinality over all strong attachments of $H$ to
$P$.

\begin{lemma}[Fan \cite{Fan}]
Let $G$ be a graph and $P$ a path of $G$. Suppose that $H$ is a
component of $G-P$ and $T=\{u_1,u_2,\ldots,u_t\}$ is a maximum
strong attachment of $H$ to $P$. Set $S=N_P(H)\backslash T \mbox{ and }s=|S|$.
Then the following statements are true:\\
(1) Every vertex in $S$ is joined to exactly one vertex in $H$.\\
(2) For each segment $P[u_i,u_{i+1}]$, $1\leq i\leq t-1$, suppose
that
$$
N_P(H)\cap V(P[u_i,u_{i+1}])=\{a_0,a_1,\ldots,a_q,a_{q+1}\},
$$
where $a_0=u_i$, $a_{q+1}=u_{i+1}$ and $a_j$, $0\leq j\leq q+1$, are
in order {along} $P$. Then there is a subscript $m$, $0\leq m\leq
q$, such that
$$
N_H(a_j)=N_H(a_0), \mbox{ for }0\leq j\leq m,
$$
and
$$
N_H(a_j)=N_H(a_{q+1}), \mbox{ for }m+1\leq j\leq q+1.
$$
Besides, if
$$
N_P(H)\cap V(P[x,u_1])=\{a_1,\ldots,a_q,a_{q+1}\},
$$
where, $a_{q+1}=u_1$, then
$$
N_H(a_j)=N_H(a_{q+1}), \mbox{ for }1\leq j\leq q+1;
$$
and if
$$
N_P(H)\cap V(P[u_t,z])=\{a_0,a_1,\ldots,a_q\},
$$
where, $a_0=u_t$, then
$$
N_H(a_j)=N_H(a_0), \mbox{ for }0\leq j\leq q.
$$
(3) If $H$ is locally $k$-connected to $P$ in $G$, then
$$
t\geq\min\{k,h+d_2\},
$$
where $h=|V(H)|$ and $d_2$ is the number of vertices in $N_P(H)$
which has at least two neighbors in $H$.
\end{lemma}

{Lemma 1 (2) is somewhat different from that in \cite{Fan}, but the proofs of them are
similar.}

For a path {$P$}, we use $l(P)$ to denote the length of
$P$.

\begin{lemma}
Let $G$ be a graph, $P$ an $(x,Y,z)$-path of $G$, where $x,z\in
V(G)$ and $Y\subset V(G)$, $H$  a component of $G-P$ and
$T=\{u_1,u_2,\ldots,u_t\}$  a maximum strong attachment of $H$ to
$P$. Set $S=N_P(H)\backslash T \mbox{ and }s=|S|$.
Suppose that $P$ is a locally longest $(x,Y,z)$-path with respect to
$H$, and $\theta=|\{x,z\}\cap N_P(H)|$. Set
$$
T_r=\{u_i\in T\backslash\{u_t\}: Y\cap V(P(u_i,u_{i+1}))=\emptyset\}
\mbox{ and }t_r=|T_r|.
$$
Then
$$
l(P)\geq\sum_{u_i\in T_r}d_H^*(u_i,u_{i+1})+2(s+t-t_r)-\theta.
$$
\end{lemma}

\begin{proof}
If $t=0$, then $s=0$ and the statement is trivially true. Suppose
now that $t\geq 1$.

Consider a segment $P[u_i,u_{i+1}]$, $1\leq i \leq t-1$. Suppose
that
$$
N_P(H)\cap V(P[u_i,u_{i+1}])=\{a_0,a_1,\ldots,a_q,a_{q+1}\},
$$
where $q=|S\cap V(P[u_i,u_{i+1}])|$, $a_0=u_i$, $a_{q+1}=u_{i+1}$,
and $a_j$, $0\leq j\leq q+1$, are in order {along} $P$.

If $Y\cap V(P(u_i,u_{i+1}))=\emptyset$, then by Lemma {1 (2)}, there is
a subscript $m$, $0\leq m\leq q$, such that
$$
N_H(a_0)=N_H(a_m) \mbox{ and } N_H(a_{q+1})=N_H(a_{m+1}).
$$
Therefore
$$
d_H^*(a_m,a_{m+1})=d_H^*(a_0,a_{q+1})=d_H^*(u_i,u_{i+1}).
$$
Since $P$ is a locally longest $(x,Y,z)$-path with respect to $H$, we have
\begin{align*}
l(P[u_i,u_{i+1}])   &\geq\sum_{j=0}^qd_H^*(a_j,a_{j+1})=d_H^*(a_m,a_{m+1})+
    \sum_{^{j=0}_{j\neq m}}^qd_H^*(a_j,a_{j+1})\\
    &=d_H^*(u_i,u_{i+1})+ \sum_{^{j=0}_{j\neq
    m}}^qd_H^*(a_j,a_{j+1}).
\end{align*}
Note that $d^*_H(a_j,a_{j+1})\geq 2$, for every $j$, $0\leq j\leq
q$, we have
$$
l(P[u_i,u_{i+1}])\geq d_H^*(u_i,u_{i+1})+2q.
$$

If $Y\cap V(P(u_i,u_{i+1}))\neq\emptyset$, then {noting} that
$l(P[a_j,a_{j+1}])\geq 2$, we have
$$
l(P[u_i,u_{i+1}])=\sum_{j=0}^ql(P[a_j,a_{j+1}])\geq 2q+2.
$$

Besides, consider the two segments $P[x,u_1]$ and $P[u_t,z]$.
Suppose that
\begin{align*}
N_P(H)\cap V(P[x,u_1])=\{a_0,a_1,\ldots,a_m\}
\end{align*}
and
\begin{align*}
N_P(H)\cap V(P[u_t,z])=\{a_{m+1},a_{m+2},\ldots,a_{q+1}\},
\end{align*}
where $m=|S\cap V(P[x,u_1])|$, $q-m=|S\cap V(P[u_t,z])|$, $a_m=u_1$,
$a_{m+1}=u_t$, and $a_j$, $0\leq j\leq q+1$ are in order {along} $P$.
Note that $l(P[x,a_0])+l(P[a_{q+1},z])\geq 2-\theta$ and
$l(P[a_j,a_{j+1}])\geq 2$, for every $0\leq j\leq q$, and $j\neq m$,
we have
$$
l(P[x,u_1])+l(P[u_t,z])\geq 2q+2-\theta.
$$

Thus summing over the {lengths} of all the segments, yields
\begin{align*}
l(P)&=l(P[x,u_1])+\sum_{i=1}^{t-1}l(P[u_i,u_{i+1}])+l(P[u_t,z])\\
    &\geq 2(|S\cap V(P[x,u_1])|+|S\cap V(P[u_t,z])|)+2-\theta\\
    &+\sum_{^{~i=1}_{u_i\in T_r}}^{t-1}(d_H^*(u_i,u_{i+1})+2|S\cap
    V(P[u_i,u_{i+1}])|)+\sum_{^{~i=1}_{u_i\notin T_r}}^{t-1}(2|S\cap
    V(P[u_i,u_{i+1}])|+2)\\
    &=\sum_{u_i\in T_r}d_H^*(u_i,u_{i+1})+2(s+t-t_r)-\theta.
\end{align*}

This ends the proof.
\end{proof}

In the following, we call a strong attached pair $\{u_j,u_{j+1}\}$
of $H$ to $P$ in $G$ \emph{transitive} if $Y\cap
V(P(u_j,u_{j+1}))=\emptyset$.

\begin{lemma}
Let $G$ be a graph {and }$P$ a path of $G$. Suppose that $H$ is a
separable component of $G-P$, $B$ is an endblock of $H$, $b$ is the
cut vertex of $H$ contained in $B$, $M=B-b$. Let
$T=\{u_1,u_2,\ldots,u_t\}$ be a maximum strong attachment of $H$ to
$P$. If $H$ is locally $k$-connected to {$P$, then}\\
(1) {$|N_P(M)\cap T|\geq\min\{k-1,m+d'_2\}$}; and\\
(2) there exist at least $\min\{k-1,m+d'_2\}$ strong attached pairs
which are joined to $M$,\\
where $m=|V(M)|$ and $d'_2$ is the number of vertices in $N_P(M)$
which has at least two neighbors in $H$.
\end{lemma}

\begin{proof}
Since $H$ is locally $k$-connected to $P$, $|V(P)|\geq k$. {It is easy} to
know that $M$ is locally $(k-1)$-connected to $P$ in the subgraph
induced by {$V(P)\cup V(M)$}. By Proposition 1, there are
$\min\{k-1,m\}$ independent edges in $E(P,M)$. Let $v_iw_i$, $1\leq
i\leq\min\{k-1,m\}$ be such edges, where $v_i\in V(P)$ and $w_i\in
V(M)$.

If $v_i$ has at least two neighbors in $H$, then by Lemma {1 (1)},
$v_i\in T$. If $v_i$ has only one neighbor $w_i$ in $H$, then by
Lemma {1 (2)}, there exists a vertex $v'_i$ (maybe $=v_i$) in $T$ which
also has only one neighbor $w_i$ in $H$. This implies that
$|N_P(M)\cap T|\geq\min\{k-1,m\}$.

Now, we prove (1) by induction on $d'_2$. If $d'_2=0$, then by the
analysis above, the assertion is true. Thus we assume that $d'_2\geq
1$.

Let $u_j$ be a vertex in $N_P(M)$ which has at least two neighbors
in $H$ ($u_j$ is of course in $T$ by Lemma {1 (1)}{{)}}. Let $G'$ be the
graph obtained from $G$ by deleting all edges from $u_j$ to $H$. By
Proposition 2, $H$ is locally $(k-1)$-connected to $P$ in $G'$.

If $u_j=u_1$ or $u_t$, or $\{u_{j-1},u_{j+1}\}$ are joined to $H$ by
two independent edges, then $T'=T\backslash\{u_j\}$ is a strong
attachment of $H$ to $P$ in $G'$. Since $u_j$ is joined to at least
two vertices of $H$ in $G$, any strong attachment of $H$ to $P$ in
$G'$ together with $u_j$ is a strong attachment of $H$ to $P$ in
$G$. Since $|T'|=t-1$, we see that $T'$ is a maximum strong
attachment of $H$ to $P$ in $G'$. By the induction hypothesis,
{$$
|N_P(M)\cap T'|\geq\min\{k-2,m+d'_2-1\}.
$$}
Therefore
{$$
|N_P(M)\cap T|\geq\min\{k-1,m+d'_2\},
$$}
as required.

If $u_j\in \{u_2,\ldots,u_{t-1}\}$, and $\{u_{j-1},u_{j+1}\}$ {are
not} joined to $H$ by two independent edges, i.e.,
$$
N_H(u_{j-1})=N_H(u_{j+1})=\{w\},
$$
for some $w\in V(H)$, then
$$
T'=T\backslash\{u_j,u_{j+1}\}=\{u_1,\ldots,u_{j-1},u_{j+2},\ldots,u_t\}
$$
is a strong attachment of $H$ to $P$ in $G'$. We prove now that $T'$
is maximum by showing that any strong attachment of $H$ to $G'$ has
cardinality at most $t-2=|T'|$.

Let $v_1,v_2$ ($\neq u_j$) be the two vertices in $N_P(H)$ {which are
closest to $u_j$ on $P$, say $v_1$ preceding, and $v_2$ following,}
$u_j$ on $P$ (but not necessarily adjacent to $u_j$ on $P$). Since
$|N_H(u_j)|\geq 2$ and by Lemma {1 (2)},
$$
N_H(v_1)=N_H(u_{j-1})=\{w\}=N_H(u_{j+1})=N_H(v_2).
$$
By the choice of $v_1$ and $v_2$, for any maximum strong attachment
$\{a_1,a_2,\ldots,a_p\}$ of $H$ to $P$ in $G'$, there is an integer
$l$, $0\leq l\leq p$, such that $v_1,v_2\in V(P[a_l,a_{l+1}])$,
where $a_0=x$ and $a_{p+1}=z$. Since $N_H(v_1)=\{w\}=N_H(v_2)$, it
follows from Lemma {1 (2)} that either $N_H(a_l)$ or
$N_H(a_{l+1})=\{w\}$. The former implies a strong attachment
$\{a_1,\ldots,a_l,u_j,v_2,a_{l+1},\ldots,a_p\}$, the latter a strong
attachment $\{a_1,\ldots,a_l,v_1,u_j,a_{l+1},\ldots,a_p\}$, of $H$
to $P$ in $G$; in either case we have that $p+2\leq t$, that is,
$p\leq t-2=|T'|$. This shows that $T'$ is a maximum strong
attachment of $H$ to $P$ in $G'$, as claimed. As before, by the
induction hypothesis,
$$
|N_P(M)\cap T'|\geq\min\{k-2,m+d'_2-1\}.
$$
Consequently
$$
|N_P(M)\cap T|\geq\min\{k-1,m+d'_2\},
$$
which completes the proof of (1).

Now we prove (2). Clearly for every vertex $u_j\in N_P(M)\cap
T\backslash\{u_t\}$, the strong attached pair $\{u_j,u_{j+1}\}$ is
joined to $M$. If {$|N_P(M)\cap
T\backslash\{u_t\}|\geq\min\{k-1,m+d'_2\}$}, then the assertion is
true. By (1), we assume that {$|N_P(M)\cap T|=\min\{k-1,m+d'_2\}$} and
$u_t\in N_P(M)\cap T$.

By Lemma {1 (3)}, $t\geq\min\{k,h+d_2\}\geq \min\{k-1,m+d'_2\}+1${. This}
implies that there exists at least one vertex in $T\backslash
N_P(M)$. We chose a vertex $u_i\in T\backslash N_P(M)$ such that
$u_{i+1}\in N_P(M)\cap T$. Then $\{u_i,u_{i+1}\}$ together with
$\{u_j,u_{j+1}\}$ for $u_j\in N_P(M)\cup T\backslash\{u_t\}$ are
$\min\{k-1,m+d'_2\}$ strong attached pairs joined to $M$.
\end{proof}

In the following, we call a strong attached pair which is joined
to $M$ a \emph{good pair} (with respect to $M$). Let
$\{u_j,u_{j+1}\}$ be a {strong attached pair}. If one of the
vertices in $\{u_j,u_{j+1}\}$ is joined to $M$, and {the other}
to $H-M$, then we call it a \emph{better pair} (with respect to
$M$); and if one of the vertices in $\{u_j,u_{j+1}\}$ is joined to
$M$, and {the other} to $H-B$, then we call it a \emph{best pair}
(with respect to $M$).

\section{Proof of Theorem 4}

In order to prove the theorem, we chose a longest $(x,Y,z)$-path $P$
{in} $G$. Clearly $|V(P)|\geq k$. Moreover, by the $k$-connectedness
of $G$, for each component $H$ of $G-P$, $H$ is locally
$k$-connected to $P$, and $P$ is a locally longest $(x,Y,z)$-path
with respect to $H$. So it is sufficient to prove that:

\begin{prop}
Let $G$ be a graph, $P$ an $(x,Y,z)$-path of $G$, where $x,z\in
V(G)$, $Y\subset V(G)$, and $|Y|=k-2$. Suppose that the average
degree of vertices in $V(G)\backslash\{x,z\}$ is $r$. If for each
component $H$ of $G-P$, $H$ is locally $k$-connected to $P$, and $P$
is a locally longest $(x,Y,z)$-path with respect to $H$, then
$l(P)\geq r$.
\end{prop}

\begin{proof}
We prove this proposition by induction on $|V(G-P)|$. If
$V(G-P)=\emptyset$, note that $r\leq |V(G)|-1$, the result is
trivially true. So we assume that $V(G-P)\neq\emptyset$. Let $H$ be
a component of $G-P$.

Let $d=|N_{P}(H)|$, $\theta=|\{x,z\}\cap N_P(H)|$ and
$N_P(H)=\{v_{1},v_{2},\ldots,v_d\}$, where $v_i$, $1\leq i\leq d$,
are in order  {along} $P$. Then, we have
$$
l(P)=l(P[x,v_{1}])+\sum_{i=1}^{d-1}l(P[v_{i},v_{i+1}])+l(P[v_d,z]).
$$
It is {easy} to know that $l(P[x,v_{1}])+l(P[v_d,z])\geq 2-\theta$
and $l(P[v_{i},v_{i+1}])\geq 2$ for $1\leq i\leq d-1$. Thus, we have
$$
l(P)\geq 2d-\theta.
$$

Note that $d\geq k$ by the {local} $k$-connectedness of $H$ to $P$
and clearly $\theta\leq 2$. If $r\leq 2k-2$, then we have $l(P)\geq
2k-2\geq r$, and the proof is complete. Thus we assume that
\begin{align}
r>2k-2.
\end{align}

Besides, if $d\geq(r+\theta)/2$, then $l(P)\geq r$, and we complete
the proof. Thus, we assume that
\begin{align}
d<(r+\theta)/2.
\end{align}

Let $T=\{u_1,u_2,\ldots,u_t\}$  be a maximum strong attachment of
$H$ to $P$. Set $S=N_{P}(H)\backslash T$ and $s=|S|$ (note that
$s+t=d$). Let {$T_r=\{u_i\in T\backslash\{u_t\}: Y\cap
V(P(u_i,u_{i+1}))=\emptyset\}$} and $t_r=|T_r|$.

Clearly, for every transitive strong attached pair
$\{u_j,u_{j+1}\}$, where $u_j\in T_r$, we have
\begin{align}
d_{H}^{*}(u_j,u_{j+1})\geq 2.
\end{align}

We distinguish two cases:

\begin{case}
$H$ is nonseparable.
\end{case}
{
Let $h=|V(H)|$ and $r'$ the average degree of vertices in $V(H)$.
If $r'h+e(P-\{x,z\},H)\leq rh$, then we consider the graph $G'$
obtained from $G$ by deleting the component $H$. Note that
\begin{align*}
\sum_{v\in V(G')\backslash\{x,z\}}{d_{G'}{(v)}}    &=
r(|V(G)|-2)-r'h-e(P-\{x,z\},H)\\
    &\geq r(|V(G)|-2)-rh\\
    &=r(|V(G')|-2).
\end{align*}
By the induction hypothesis, we have $l(P)\geq r$, and the proof is
complete. Thus we assume that
\begin{align}
r'h+e(P-\{x,z\},H)>rh
\end{align}
}
We use $d_{1}$ to denote the number of vertices in $N_{P}(H)$ which
have only one neighbor in $V(H)$, $d_2=d-d_1$, $\theta_{1}$ to
denote the number of vertices in $\{x,z\}$ which have only one
neighbor in $V(H)$ and $\theta_{2}=\theta-\theta_1$.

Clearly,
$$
r'h\leq h(h-1+d_{2})+d_{1} \mbox{ and } e(P-\{x,z\},H)\leq
h(d_{2}-\theta_{2})+d_{1}-\theta_{1}.
$$
Thus, by {(4)}, we have
$$
h(h-1+2d_{2}-\theta_2)+2d_{1}-\theta_{1}\geq r'h+e(P-\{x,z\},H)>rh.
$$
Note that $d_1=d-d_2$ and $\theta_1=\theta-\theta_2$, {we have}
$$
h(h-1+2d_{2}-\theta_2)+2d-2d_2-\theta+\theta_2\geq rh.
$$
By (2), we have
$$
h(h-1+2d_{2}-\theta_2)+(r+\theta)-2d_2-\theta+\theta_2>rh.
$$
Thus
$$
(h-1)(h+2d_{2}-r-\theta_{2})>0.
$$
This implies that $h\geq 2$ and $h+2d_{2}>r+\theta_{2}\geq r$, and
then $2h+2d_2>r+2$. By (1), {we have} $2h+2d_2>2k$, that is
\begin{align}
h+d_2>k.
\end{align}

By (5) and Lemma {1 (3)}, $t\geq k$. Since $|Y|\leq k-2$, there exists
at least one transitive strong attached pair $(u_p,u_{p+1})$ in $T$,
where $u_p\in T_r$.

Let $G'$ be the subgraph induced by $V(H)\cup\{u_p,u_{p+1}\}$. If
$u_pu_{p+1}\notin E(G)$, we add the edge $u_pu_{p+1}$ in $G'$. Thus
$G'$ is 2-connected and
\begin{align*}
\sum_{v\in V(G')\backslash\{u_p,u_{p+1}\}}d_{G'}(v)   &=\sum_{v\in
V(H)}d(v)-e(N_P(H)\backslash\{u_p,u_{p+1}\},H)\\
    &=r'h-e(N_P(H)\backslash\{u_p,u_{p+1}\},H)\\
    &\geq rh-e(P-\{x,z\},H)-e(N_P(H)\backslash\{u_p,u_{p+1}\},H).
\end{align*}

Note that
\begin{align*}
    &e(P-\{x,z\},H)\leq (s+t-\theta)h, \mbox{ and }\\
    &e(N_P(H)\backslash\{u_p,u_{p+1}\},H)\leq (s+t-2)h,
\end{align*}
we have
\begin{align*}
\sum_{v\in V(G')\backslash\{u_p,u_{p+1}\}}d_{G'}(v)   &\geq
rh-(s+t-\theta)h-(s+t-2)h\\
    &=(r-2s-2t+\theta+2)h.
\end{align*}

By Theorem 3, $G'$ contains a $(u_p,u_{p+1})$-path of length at
least $r-2s-2t+\theta+2$, which implies that
\begin{align}
d_{H}^{*}(u_p,u_{p+1})\geq r-2s-2t+\theta+2.
\end{align}

Substituting (6) for $d_{H}^{*}(u_p,u_{p+1})$ in Lemma 2 and {(3)} for
the other terms, we have
$$
l(P)\geq (r-2s-2t+\theta+2)+2(t_{r}-1)+2(s+t-t_r)-\theta\geq r.
$$

\begin{case}
$H$ is separable.
\end{case}

{Let $B$ be an endblock of $H$, $b$ the cut vertex of $H$
contained in $B$, $M=B-b$, $m=|V(M)|$, and $r''$ the average
degree of the vertices in $V(M)$.}

If $r''m+e(P-\{x,z\},M)+d_{M}(b)\leq rm$, then we consider the graph
$G'$ obtained from $G$ by contracting $B$. Let $H'$ be the component
of $G'-P$ obtained from $H$ by contracting $B$. By Proposition 3,
$H'$ is locally $k$-connected to $P$. Clearly $P$ is a locally
longest $(x,Y,z)$-path with respect to $H'$, and
\begin{align*}
\sum_{v\in V(G')\backslash\{x,z\}}d_{G'}(v) &\geq\sum_{v\in
V(G)\backslash\{x,z\}}d(v)-r''m-e(P-\{x,z\},M)-d_{M}(b)\\
    &\geq r(|V(G)|-2)-rm\\
    &=r(|V(G')|-2).
\end{align*}
By the induction hypothesis, $l(P)\geq r$, and the proof is
complete. Thus we assume that
\begin{align}
r''m+e(P-\{x,z\},M)+d_{M}(b)>rm.
\end{align}

Let $d'_0=|N_P(H)\backslash N_P(M)|$, $d'_1$ be the number of
vertices in $N_P(M)$ which have only one neighbor in $V(H)$,
$d'_2=d-d'_0-d'_1$; $\theta'_0=|\{x,z\}\cap N_P(H)\backslash
N_P(M)|$, $\theta'_1$ be the number of vertices in $\{x,z\}\cap
N_P(M)$ which have only one neighbor in $V(H)$ and
$\theta'_2=\theta-\theta'_0-\theta'_1$.

Now we prove that
\begin{align}
m+d'_2\geq k-1.
\end{align}

Let $B'$ be an endblock of $H$ other than $B$, $b'$ the cut
vertex of $H$ contained in $B'$, $M'=B'-b'$ and $m'=|V(M')|$.

By the {local} $k$-connectedness of $H$ to $P$, $|N_P(M')|\geq k-1$.
If $|N_P(M')\backslash N_P(M)|\leq m$, then $d'_2\geq |N_P(M)\cap
N_P(M')|\geq k-1-m$, and $m+d'_2\geq k-1$, and (8) holds. Thus we
assume that $|N_P(M')\backslash N_P(M)|\geq m+1$. So we have
\begin{align}
d'_0\geq m+1.
\end{align}

Clearly,
\begin{align*}
    &r''m\leq m(m+d'_2)+d'_1,\\
    &e(P-\{x,z\},M)\leq m(d'_2-\theta'_2)+d'_{1}-\theta'_1, \mbox{ and }\\
    &d_M(b)\leq m.
\end{align*}

Thus, by (7),
$$
m(m+2d'_2+1-\theta'_2)+2d'_1-\theta'_1\geq
r''m+e(P-\{x,z\},M)+d_{M}(b)>rm.
$$
Note that $d'_1=d-d'_0-d'_2$ and
$\theta'_1=\theta-\theta'_0-\theta'_2$, {we have}
$$
m(m+2d'_{2}+1-\theta'_2)+2d-2d'_0-2d'_2-\theta+\theta'_0+\theta'_2>rm.
$$
By (2) and (9), we have
$$
m(m+2d'_{2}+1-\theta'_2)+(r+\theta)-2(m+1)-2d'_2-\theta+\theta'_0+\theta'_2>rm.
$$
Thus
$$
(m-1)(m+2d'_{2}-r-\theta'_{2})>2-\theta'_0\geq 0.
$$
This implies that $m\geq 2$ and $m+2d'_{2}>r+\theta'_{2}\geq r$, and
then $2m+2d'_2>r+2$. By (1), $2m+2d'_2>2k$, that is $m+d'_2>k$, and
(8) holds.

By Lemma {3 (2)}, there exist at least $k-1$ good pairs with respect to
$M$. Since $|Y|=k-2$, there exists at least one transitive good pair
$\{u_p,u_{p+1}\}$ with respect to $M$. Similarly there exists at
least one transitive good pair $\{u_q,u_{q+1}\}$ with respect to
$M'$.

First we assume that there is a transitive best pair with respect to
$M$ or $M'$. Without loss of generality, we assume that
$\{u_p,u_{p+1}\}$ is a best pair, where $u_p\in N_P(M)$ and
$u_{p+1}\in N_P(H-B)$. Consider the subgraph $G'$ induced by
$V(B)\cup\{u_p\}$. If $u_pb\notin E(G)$, we add the edge $u_pb$ in
$G'$. Thus $G'$ is 2-connected and
\begin{align*}
\sum_{v\in V(G')\backslash\{u_p,b\}}d_{G'}(v)    &=\sum_{v\in
V(M)}d(v)-e(N_P(H)\backslash\{u_p\},M)\\
    &=r''m-e(N_P(H)\backslash\{u_p\},M)\\
    &\geq rm-e(P-\{x,z\},M)-d_M(b)-e(N_P(H)\backslash\{u_p\},M).
\end{align*}

Note that
\begin{align*}
    &e(P-\{x,z\},M)\leq (s+t-\theta)m,\\
    &d_M(b)\leq m, \mbox{ and }\\
    &e(N_P(H)\backslash\{u_p\},M)\leq (s+t-1)m,
\end{align*}
we have
\begin{align*}
\sum_{v\in V(G')\backslash\{u_p,b\}}d_{G'}(v)   &\geq
rm-(s+t-\theta)m-m-(s+t-1)m\\
    &=(r-2s-2t+\theta)m.
\end{align*}

By Theorem 3, $G'$ contains a $(u_p,b)$-path of length at least
$r-2s-2t+\theta$. It is {clear} that there is a $(b,u_{p+1})$-path
in $H-B$ of length at least 2, which implies that
\begin{align}
d_{H}^{*}(u_p,u_{p+1})\geq r-2s-2t+\theta+2.
\end{align}

Substituting (10) for $d_{H}^{*}(u_p,u_{p+1})$ in Lemma 2 and {(3)}
for the other terms, we have
$$
l(P)\geq (r-2s-2t+\theta+2)+2(t_{r}-1)+2(s+t-t_r)-\theta\geq r,
$$
as required.

So, we assume that there are no transitive best pairs with respect
to $M$ or $M'$.

Now we assume that there is a transitive better pair (but not best
pair) with respect to $M$ or $M'$. Without loss of generality, we
assume that $\{u_p,u_{p+1}\}$ is a better pair, where $u_p\in
N_P(M)$ and $u_{p+1}\in N_P(b)$. Consider the subgraph $G'$ induced
by $V(B)\cup\{u_p\}$. If $u_pb\notin E(G)$, we add the edge $u_pb$
in $G'$. Thus $G'$ is 2-connected and
\begin{align*}
\sum_{v\in V(G')\backslash\{u_p,b\}}d_{G'}(v)\geq
rm-e(P-\{x,z\},M)-d_M(b)-e(N_P(H)\backslash\{u_p\},M).
\end{align*}

Note that
\begin{align*}
    &e(P-\{x,z\},M)\leq (s+t-\theta)m, \mbox{ and }\\
    &d_M(b)\leq m,
\end{align*}
and since at least one vertex of $u_q$ and $u_{q+1}$ is not joined
to $M$ (otherwise, $\{u_q,u_{q+1}\}$ will be a best pair), we have
$$
e(N_P(H)\backslash\{u_p\},M)\leq (s+t-2)m.
$$
Thus we have
\begin{align*}
\sum_{v\in V(G')\backslash\{u_p,b\}}d_{G'}(v)   &\geq
rm-(s+t-\theta)m-m-(s+t-2)m\\
    &=(r-2s-2t+\theta+1)m.
\end{align*}

By Theorem 3, $G'$ contains a $(u_p,b)$-path of length at least
$r-2s-2t+\theta+1$, and then, by $bu_{p+1}\in E(G)$,
\begin{align*}
d_{H}^{*}(u_p,u_{p+1})\geq r-2s-2t+\theta+2.
\end{align*}
Thus we also have $l(P)\geq r$.

So, we assume that there are no transitive better pairs with respect
to $M$ or $M'$. Thus $\{u_p,u_{p+1}\}$ and $\{u_q,u_{q+1}\}$ are two
distinct strong attached pairs.

If $m=1$, then $\{u_p,u_{p+1}\}$ will be a better pair with respect
to $M$. Thus we assume that $m\geq 2$.

If $m=2$, then $B$ is a triangle, and $d_H^*(u_p,u_{p+1})=4$. Since
$\{u_p,u_{p+1}\}$ is not a better pair, we have that $u_p\in
N_P(M)$. {Similar} to the analysis above, we have $d_H^*(u_p,b)\geq
r-2s-2t+\theta+1$. But $d_H^*(u_p,b)=3$, we have
$$
d_H^*(u_p,u_{p+1})=4\geq r-2s-2t+\theta+2.
$$
Then $l(P)\geq r$.

So we {assume} that
\begin{align}
m\geq 3, \mbox{ and similarly, } m'\geq 3.
\end{align}

{It is easy} to know that $d_H^*(u_p,u_{p+1})\geq 4$. Thus if
$r-2s-2t+\theta\leq 2$, we will have
$$
d_H^*(u_p,u_{p+1})\geq r-2s-2t+\theta+2,
$$
and then $l(P)\geq r$. So we assume that
\begin{align}
r-2s-2t+\theta\geq 2.
\end{align}

Note that $u_p$ and $u_{p+1}$ are joined to $B$ by two independent
edges. Consider the subgraph $G'$ induced by
$V(B)\cup\{u_p,u_{p+1}\}$. If $u_pu_{p+1}\notin E(G)$, we add the
edge $u_pu_{p+1}$ in $G'$. Thus $G'$ is 2-connected and
\begin{align*}
    &\sum_{v\in V(G')\backslash\{u_p,u_{p+1}\}}d_{G'}(v)\\
    &=\sum_{v\in
V(M)}d(v)-e(N_P(H)\backslash\{u_p,u_{p+1}\},M)+d_M(b)+|\{u_p,u_{p+1}\}\cap N(b)|\\
    &=r''m+d_M(b)-e(N_P(H)\backslash\{u_p,u_{p+1}\},M)+|\{u_p,u_{p+1}\}\cap N(b)|\\
    &\geq rm-e(P-\{x,z\},M)-e(N_P(H)\backslash\{u_p,u_{p+1}\},M).
\end{align*}

Note that
\begin{align*}
    &e(P-\{x,z\},M)\leq (s+t-\theta)m, \mbox{ and }\\
    &e(N_P(H)\backslash\{u_p,u_{p+1}\},M)\leq (s+t-2)m,
\end{align*}
we have
\begin{align*}
\sum_{v\in V(G')\backslash\{u_p,u_{p+1}\}}d_{G'}(v)   &\geq
rm-(s+t-\theta)m-(s+t-2)m\\
    &=(r-2s-2t+\theta+2)m.
\end{align*}

By Theorem 3, $G'$ contains a $(u_p,u_{p+1})$-path of length at
least $(r-2s-2t+\theta+2)m/({1+m})$, which implies that
\begin{align*}
d_{H}^{*}(u_p,u_{p+1})  &\geq (r-2s-2t+\theta+2)\frac{m}{1+m}\\
    &\geq \frac{3}{4}(r-2s-2t+\theta+2).
\end{align*}
(note that $m\geq 3$), and similarly,
\begin{align*}
d_{H}^{*}(u_q,u_{q+1})\geq \frac{3}{4}(r-2s-2t+\theta+2).
\end{align*}
Then by (12),
\begin{align*}
d_{H}^{*}(u_p,u_{p+1})  &+d_{H}^{*}(u_q,u_{q+1})\\
    &\geq\frac{3}{2}(r-2s-2t+\theta+2)\\
    &=(r-2s-2t+\theta+2)+\frac{1}{2}(r-2s-2t+\theta+2)\\
    &\geq r-2s-2t+\theta+4.
\end{align*}
Thus{, by Lemma 2, we have}
$$
l(P)\geq (r-2s-2t+\theta+4)+2(t_{r}-2)+2(s+t-t_r)-\theta\geq r.
$$

The proof is complete.
\end{proof}

\section{Proof of Theorem 8}

By the $k$-connectedness of $G$, it contains a $Y$-cycle. If
$2e(G)/(n-1)\leq 3$, then the result is trivially true. Thus we
assume that $2e(G)/(n-1)>3$.

We chose a vertex $y\in Y$, and construct a graph $G'$ such that
$V(G')=V(G)\cup\{y'\}$, where $y'\notin V(G)$ and
$E(G')=E(G)\cup\{vy': v\in N_{G}(y)\}$. {Clearly}, $G'$ is
$k$-connected. Besides, we have that
$$
e(G')=e(G)+d_{G}(y) \mbox{ and } d_{G'}(y)=d_{G'}(y')=d_G(y),
$$
and the order of $G'$ is $n+1$. Now, {by Theorem 4}, there exists a
$(y,Y\backslash\{y\},y')$-path $P$ of length at least
\begin{align*}
\frac{2e(G')-d_{G'}(y)-d_{G'}(y')}{(n+1)-2}=\frac{2(e(G)+d_{G}(y))-2d_{G}(y)}{n-1}=\frac{2e(G)}{n-1}.
\end{align*}

Let $uy'$ be the last edge of $P$, then $uy\in E(G)$ and
$C=P[y,u]uy$ is a cycle of $G$ passing through all the vertices in
$Y$ of length at least $2e(G)/(n-1)$, which completes the proof.
{\hfill$\Box$}

\end{document}